\documentclass{elsarticle}
\usepackage{amsmath,amssymb,latexsym,amsmath,enumerate,verbatim,amsfonts,mathrsfs,bm}
\newtheorem{definition}{Definition}[section]
\newtheorem{theorem}{Theorem}[section]

\newtheorem{lemma}{Lemma}[section]
\newtheorem{example}{Example}[section]

\numberwithin{equation}{section}
\begin{document}
	
\title{Effective Second-Harmonic Generation Coefficient  and C-eigenvalue of  Nonlinear Susceptibility Tensors}

\author[1]{Die Xiao}
\author[1]{Yisheng Song\corref{mycorrespondingauthor}}\cortext[mycorrespondingauthor]{Corresponding author.  E-mail address: yisheng.song@cqnu.edu.cn (Yisheng Song).}

\affiliation[1]{organization={School of Mathematical Sciences, Chongqing Normal University},
	city={Chongqing},
	postcode={401331}, 
	country={P.R. China}\\ { Email: 13183937768@163.com (Xiao);  yisheng.song@cqnu.edu.cn  (Song)}}

\begin{abstract} The effective second-harmonic generation (SHG) coefficient is a crucial data that quantifies the efficiency of transforming fundamental frequency light into its second harmonic.  With the help of the symmetry of nonlinear optical susceptibility tensors,   we mainly discuss  the computability of such a effective SHG coefficient in uniaxial crystals.  For one thing, the calculation of  effective SHG coefficient is  converted into the optimization models with some geometric constraints by means of  the peculiarity of fundamental frequency light. Secondly, the number of variables of  such  maximum models  are cutted in half to $2$ to  calculate it easier, and a comparison between the effective SHG coefficient and C-eigenvalue of susceptibility tensor is given also. Finally, some examples of typical crystal classes are presented to verify the correctness and broader applicabilities of the theoretical results.
\end{abstract}

\maketitle
\section{Introduction}
In 1961, the first nonlinear optical effect, Second Harmonic Generation (SHG)  is discovered \cite{FHPW1961}. For decades, as the most fundamental second-order nonlinear optical effect, it serves as a core technology for laser wavelength conversion and plays an irreplaceable role in various scientific domains such as all-solid-state lasers, physics, chemistry, materialogy, quantum communication, and biological imaging \cite{B2020, CLH1992, KBS2014,LT2014}. In nonlinear optics, the fundamental frequency light excites the crystal to generate ``second-order polarization harmonic," which may be written by the nonlinear susceptibility tensor $\overleftrightarrow{\bm\chi}^{(2)}=(\chi_{ijk}^{(2)})$ \cite{B2020,CLH1992,J1970}. Its expression  may be said as,  \begin{equation}\label{eq:11}\
	\vec{\bm P}^{(2)}=\varepsilon_0\overleftrightarrow{\bm\chi}^{(2)}\vec{\bm E}\vec{\bm E},
\end{equation} where $\vec{\bm P}^{(2)}$ is the second-order nonlinear polarization, $\varepsilon_0$ is the vacuum  permittivity, and $\vec{\bm E}$
is the electric field of the fundamental beam.  Due to the dispersion characteristics of crystals, this generated ``second-order polarization harmonic" cannot keep up with the produced frequency-doubled light, leading to a phase mismatch. Then the phase matching is required to compensate for this, and hence, four types of phase matching are involved, resulting in four different effective SHG coefficients,
\begin{align}\label{eq:12}
	\chi_{\text{eff}}^{oo-e}=&\overleftrightarrow{\bm\chi}^{(2)}{\bm b}{\bm a}{\bm a}=\sum_{i,j,k=1}^3\chi_{ijk}^{(2)}b_ia_ja_k;\\
\label{eq:13}
\chi_{\text{eff}}^{ee-o}=&\overleftrightarrow{\bm\chi}^{(2)}{\bm a}{\bm b}{\bm b}=\sum_{i,j,k=1}^3\chi_{ijk}^{(2)}a_ib_jb_k; \\
\label{eq:14}
\chi_{\text{eff}}^{oe-e}=&\overleftrightarrow{\bm\chi}^{(2)}{\bm b}{\bm a}{\bm b}=\sum_{i,j,k=1}^3\chi_{ijk}^{(2)}b_ia_jb_k;\\
\label{eq:15}
\chi_{\text{eff}}^{eo-o}=&\overleftrightarrow{\bm\chi}^{(2)}{\bm a}{\bm b}{\bm a}=\sum_{i,j,k=1}^3\chi_{ijk}^{(2)}a_ib_ja_k,
\end{align}
here, ${\bm a}$ is a vector with its element $a_i$,  the ratio between the component of basic-frequency light $o$ photoelectric field and its amplitude, ${\bm b}$ is a vector with its element $b_i$,  the ratio between the component of basic-frequency light $e$ photoelectric field and its amplitude. The effective SHG coefficient is a core parameter for evaluating the nonlinear optical performance of materials. Its value depends not only on the microscopic nonlinear polarization properties of the crystal itself but also on the phase matching conditions in three-wave interactions, the polarization direction of the optical field, and the crystal orientation. It directly determines the upper limit of harmonic conversion efficiency, which need solve maximum of the above effective SHG coefficients. Therefore, accurately calculating the effective SHG coefficient is a prerequisite and key step for designing new nonlinear optical crystals, optimizing devices, and evaluating performance \cite{CJ2000, CY2005,WC1992,ZYZ2024}.

 The nonlinear susceptibility tensor $\overleftrightarrow{\bm\chi}^{(2)}=(\chi_{ijk}^{(2)})$ with the symmetry of its last two indices  is a class of piezoelectric-type tensors  \cite{B2020,CLH1992}. So, solving effective SHG coefficients may be obtained by means of some methods and algorithms of piezoelectric-type tensors. This class of tensors may trace back to 1880, the piezoelectric effect discovered by Curie brothers \cite{CC1880}. Until 1910,  Voigt \cite{V1966} first rigorously defined the piezoelectric constants using tensor analysis, establishing the mathematical formulation of the piezoelectric tensor and systematically expressing the relationship between mechanical and electrical quantities in tensor form. Chen-J\'akli-Qi \cite{CJQ2023} introduced the concepts of  C-eigenvalues, and showed concrete physical meanings of the largest C-eigenvalues and C-eigenvectors  in iezoelectric effect and converse piezoelectric effect.
 \begin{definition}[Piezoelectric-type Tensor]
 	Let  $\mathscr{T}= (T_{ijk}) \in \mathbb{R}^{n \times n \times n}$ be a third-order real tensor. If its components satisfy the symmetry condition,
 	\[
 	T_{ijk}=T_{ikj},\quad \forall i,j,k\in \{1,2,\cdots,n\}
 	\]
 	then$ \mathscr{T}$is called a \textbf{piezoelectric-type tensor}.  If there exist a real number \( \lambda \in \mathbb{R} \) and unit vectors \( {\bm x}, {\bm y} \in \mathbb{R}^n \) such that the following system of equations holds:
 	$$
 		\mathscr{T}{\bm y}{\bm y}=\lambda {\bm x},  \ \
 		\mathscr{T}{\bm x}{\bm y}=\lambda {\bm y} ,\ \
 		{\bm x}^\top {\bm x}=1 ,\ \ 	{\bm y}^\top  {\bm y}=1,
 	$$
 	where \( (\mathscr{T}{\bm y}{\bm y})_i = \sum\limits_{j,k=1}^n T_{ijk} y_j y_k \) and \( (\mathscr{T}{\bm x}{\bm y})_k = \sum\limits_{i,j=1}^n T_{ijk} x_i y_j \), then \( \lambda \) is called a \textbf{C-eigenvalue} of \( \mathscr{T} \), and \( {\bm x}, {\bm y} \) are the corresponding \textbf{left and right C-eigenvectors}, respectively.
 \end{definition} 
 
 Over the years that followed, the C-eigenvalue gradually became one of  important concepts in tensor analysis and have been widely applied in some research areas such as piezoelectric materials  and liquid crystal physics \cite{CJQ2023,CQV2018,G2016,LLJ2025,N1985}. Some different methods of calculating the largest  C-eigenvalue of a piezoelectric-type tensor had extensively studied and they have yielded some various insights, and moreover,  these methods have been well established in the scientific community \cite{LY2021,LLJ2025,LM2022,WCW2020,WCW2023,YL2022,ZLS2023,ZL2022}. Also many works  focused on  identifying inclusion sets for C-eigenvalues such as Refs. \cite{CCW2019,HLX2021,LLL2019,LLK2014,LZL2016,LL2016,LYL2021,WCW2020,WYS2022} and others.   Recently,  Wang \cite{W2025} solved  the analytical expression of C-eigenvalues and C-eigenpairs of a three-dimensional piezoelectric-type tensor by putting its components  in categories. 

In this paper,  we mainly discuss the calculations of mathematical models for the effective SHG coefficients in uniaxial crystals, and furthermore, to  compute it easier, we establish how to diminish the number of variables of  such  optimization models.   The relationship between the effective SHG coefficients and the largest C-eigenvalue of susceptibility tensor is showed also. Finally, some examples of typical crystal classes are provided to test our conclusions.

\section{Optimization Models of Effective SHG Coefficients}

In uniaxial crystals, a basic-frequency light field is divided into isotropic ordinary light (o-light) and anisotropic extraordinary light (e-light). Based on the phase matching condition, four effective SHG coefficients are distinguished two different types of interaction: the cases $\chi_{\text{eff}}^{oo-e}$ \eqref{eq:12} and $\chi_{\text{eff}}^{ee-o}$ \eqref{eq:13} are called type-I phase matching, the cases $\chi_{\text{eff}}^{oe-e}$ \eqref{eq:14} and $\chi_{\text{eff}}^{eo-o}$ \eqref{eq:15}  are called type-II phase matching. If all indices of nonlinear susceptibility tensor  can arbitrarily be permuted, i.e., $$ \chi_{ijk}^{(2)} =\chi_{jki}^{(2)} \mbox{ and } \chi_{ijk}^{(2)} =\chi_{ikj}^{(2)},$$ then such a  susceptibility tensor is called \textbf{Kleinman symmetry} \cite{K1962}. That is,  the tensor $\overleftrightarrow{\bm\chi}^{(2)}=(\chi_{ijk}^{(2)})$ is \textbf{symmetry}.  Clearly, if $\overleftrightarrow{\bm\chi}^{(2)}$ is Kleinman symmetry, then $$\chi_{\text{eff}}^{oo-e}=\chi_{\text{eff}}^{eo-o}\mbox{ and }\chi_{\text{eff}}^{ee-o}=\chi_{\text{eff}}^{oe-e}.$$

In uniaxial crystals, we choose a crystal coordinate system such that the light propagation direction is collinear with the $x_3  $-axis. Then ${\bm a}$  is  a unit vector with its element $a_i$,   the ratio between the component of basic-frequency o-light electric field and its amplitude and ${\bm b}$ is a vector with its element $b_i$,  the ratio between the component of basic-frequency light $e$ photoelectric field and its amplitude, 
\begin{equation}\label{eq:21}
{\bm a}=(\sin\varphi,-\cos\varphi,0)^\top,
\end{equation}
\begin{equation}\label{eq:22}
{\bm b}=(-\cos\theta\cos\varphi,-\cos\theta\sin\varphi,\sin\theta)^\top,
\end{equation}
where $\theta$ is the angle between the wavevector  and the optic axis,  $\varphi$ is the angle between projection of the wavevector  to $x_1\times x_2$ coordinate plane  and the $x_1$- axis.
Then two vectors satisfy the geometric constraints,\begin{equation}\label{eq:23}{\bm a}^\top{\bm a}={\bm b}^\top{\bm b}=1\mbox{ and }{\bm a}^\top{\bm b}=a_1b_1+a_2b_2=0.\end{equation}

Let $z_1=\cos\theta$ and $z_2=\sin\theta$. Then  from \eqref{eq:22}, it follows that
\begin{equation}\label{eq:24}
	{\bm b}=(a_2z_1,-a_1z_1,z_2)^\top.
\end{equation}
Assume that $\mathscr{T}=\overleftrightarrow{\bm\chi}^{(2)}\in\mathbb{R}^{3\times3\times3}$ is a nonlinear susceptibility tensor.  Then based on Kleinman symmetry, the unified form of the effective SHG coefficient is the trilinear tensor contraction:
\[
d^1_{\text{eff}}({\bm a}, {\bm b})=\mathscr{T}{\bm b}{\bm a}{\bm a}=\sum_{i,j,k=1}^3T_{ijk}b_ia_ja_k,
\]
\[
d^2_{\text{eff}}({\bm b}, {\bm a})=\mathscr{T}{\bm a}{\bm b}{\bm b}=\sum_{i,j,k=1}^3T_{ijk}a_ib_jb_k.
\]
So, the optimization models given by solving the largest effective SHG coefficient may turn into the following forms, 
\begin{equation}\label{eq:25}
	\max\{\mathscr{T}{\bm b}{\bm a}{\bm a}: \ {\bm a}^\top{\bm a}={\bm b}^\top{\bm b}=1, a_3=0\mbox{ and }a_1b_1+a_2b_2=0 \};
\end{equation}
\begin{equation}\label{eq:26}
\max\{\mathscr{T}{\bm a}{\bm b}{\bm b}: \ {\bm a}^\top{\bm a}={\bm b}^\top{\bm b}=1, a_3=0\mbox{ and }a_1b_1+a_2b_2=0 \};
\end{equation} or
\begin{equation}\label{eq:27}
	d_{\text{eff}}^1=\max\{\mathscr{T}{\bm b}{\bm a}{\bm a}: \ a_1^2+a_2^2=z_1^2+z_2^2=1\mbox{ and }a_3=0 \};
\end{equation}
\begin{equation}\label{eq:28}
d_{\text{eff}}^2=	\max\{\mathscr{T}{\bm a}{\bm b}{\bm b}: \ a_1^2+a_2^2=z_1^2+z_2^2=1\mbox{ and }a_3=0 \}.
\end{equation}

Chen-J\'akli-Qi \cite{CJQ2023} presented  the following nice properites of the (largest) C-eigenvalue of Piezoelectric-type Tensor, and  the largest C-eigenvalue  is maximum value of some optimized model on a unit sphere. The physical significance of the C-eigenvalue problem is that it describes the maximum ability of a piezoelectric material to generate an electrical response under a unit mechanical stimulus. The largest C-eigenvalue corresponds to the optimal electromechanical coupling direction, where the piezoelectric performance of the material reaches its peak.
\begin{theorem}[C-eigenvalues of Piezoelectric-type Tensor] \cite{CJQ2023} \label{thm:21} Let  $\mathscr{T}= (T_{ijk})$ be a  piezoelectric-type tensor. Then the following conclusions holds.
\begin{itemize}
	\item [(1)]  There is at least a C-eigenvalue of $\mathscr{T}$ with associated left and right C-eigenvectors;
	\item[(2)] Assume that $\lambda$ is a C-eigenvalue  of $\mathscr{T}$ with its left and right C-eigenvectors  ${\bm x}$ and $ {\bm y}$. Then
	$$\lambda=\mathscr{T}{\bm x}{\bm y}{\bm y}=\sum_{i,j,k=1}^nx_iy_jy_k.$$
	Moreover,  $(\lambda, {\bm x}, -{\bm y})$, $(-\lambda, -{\bm x}, {\bm y})$, $(-\lambda, -{\bm x}, -{\bm y})$ are also C-eigenvalues and
	their associated left and right  C-eigenvectors of $\mathscr{T}$;
	\item[(3)] If the triples $(\tilde{\lambda}, \tilde{{\bm x}}, \tilde{{\bm y}})$  is made up of the largest C-eigenvalue of $\mathscr{T}$ and its associated left and right C-eigenvectors, then
	$$\tilde{\lambda}=\mathscr{T}\tilde{{\bm x}}\tilde{{\bm y}}\tilde{{\bm y}}= \max \{\mathscr{T}{\bm x}{\bm y}{\bm y}: {\bm x}^\top {\bm x}=1 ,\ 	{\bm y}^\top  {\bm y}=1\}. $$
	Furthermore, $\tilde{\lambda}\tilde{{\bm x}}\circ\tilde{{\bm y}}\circ\tilde{{\bm y}}$ forms the best rank-one piezoelectric-type approximation
	of $\mathscr{T}$, here ${\bm x}\circ{\bm y}$ is outer product (or tensor product) of ${\bm x}$ and ${\bm y}$.
\end{itemize}
\end{theorem}

As is well-known, on the unit sphere, the maximum of a quadratic form is the largest eigenvalue of its coefficient matrix.   Then for $2\times2$ matrix, the following conclusions are obvious.

\begin{lemma}\label{lem:21}
	Let $M=(m_{ij})$ be a $2\times2$ symmetry matrix. Then 
	$$\max\{{\bm x}^\top M{\bm x}: x_1^2+x_2^2=1\} = \dfrac12\left(m_{11}+m_{22}+\sqrt{(m_{11}-m_{22})^2+4m_{12}^2}\right).$$
\end{lemma}

\section{The Largest Effective SHG Coefficient}
Let $\mathscr{T}=\overleftrightarrow{\bm\chi}^{(2)}\in\mathbb{R}^{3\times3\times3}$ be a nonlinear susceptibility tensor in uniaxial crystals. Assume the vectors ${\bm a}$ and $ {\bm b}$ are definited by Eqs.\eqref{eq:21} and  \eqref{eq:22} or \eqref{eq:24}, respectively. Then we have  \begin{align} d^1_{\text{eff}}({\bm a}, {\bm b})=&\mathscr{T}{\bm b}{\bm a}{\bm a}=\sum_{i,j,k=1}^3T_{ijk}b_ia_ja_k\nonumber\\
=&T_{111}z_1a_1^2a_2+2T_{112}z_1a_1a^2_2+T_{122}z_1a^3_2\nonumber\\&-T_{211}z_1a_1^3-2T_{212}z_1a^2_1a_2-T_{222}z_1a_1a^2_2\nonumber\\&+T_{311}z_2a_1^2+2T_{312}z_2a_1a_2+2T_{322}z_2a^2_2\nonumber\\
=&z_1\left(a_2Q_1({\bm a})-a_1Q_2({\bm a})\right)+z_2Q_3({\bm a}),\label{eq:31}\end{align}
where  $Q_i({\bm a})=\sum\limits_{j,k=1}^2 T_{ijk}a_ja_k$ is a quadratic form about ${\bm a}$, and 
\begin{align}d^2_{\text{eff}}({\bm b}, {\bm a})=&\mathscr{T}{\bm a}{\bm b}{\bm b}=\sum_{i,j,k=1}^3T_{ijk}a_ib_jb_k\nonumber\\
=&T_{111}z^2_1a_1a^2_2-2T_{112}z^2_1a^2_1a_2+T_{122}z^2_1a^3_1\nonumber\\&+2T_{113}z_1z_2a_1a_2-2T_{123}z_1z_2a^2_1+T_{133}z^2_2a_1\nonumber\\&+T_{211}z^2_1a^3_2-2T_{212}z^2_1a_1a^2_2+T_{222}z^2_1a^2_1a_2\nonumber\\&+2T_{213}z_1z_2a^2_2-2T_{223}z_1z_2a_1a_2+T_{233}z^2_2a_2\nonumber\\
=&z^2_1R_{11}({\bm a})+2z_1z_2R_{12}({\bm a})+z^2_2R_{22}({\bm a}),\label{eq:32}\end{align} here 
$$R_{11}({\bm a}))=T_{111}a_1a^2_2+(T_{222}-2T_{112})a^2_1a_2+(T_{122}-2T_{212})a^3_1+T_{211}a^3_2,$$
$$R_{12}({\bm a}))=(T_{113}-T_{223})a_1a_2-T_{123}a^2_1+T_{213}a^2_2, R_{22}({\bm a})=T_{133}a_1+T_{233}a_2.$$

\begin{theorem}\label{thm:31}
	Let $\mathscr{T}$ be a nonlinear susceptibility tensor.  Then in uniaxial crystals, the maximum model \eqref{eq:27} of effective SHG coefficient may cut the number of variables in half to $2$. That is, 
\begin{equation}\label{eq:33}
	d_{\text{eff}}^1=\max\left\{\sqrt{\left(a_2Q_1({\bm a})-a_1Q_2({\bm a})\right)^2+(Q_3({\bm a}))^2}: a_1^2+a_2^2=1\right\}.
\end{equation}
	\end{theorem}

{\bf Proof.} It follows from the constrained optimization model \eqref{eq:27} that we may construct the Lagrangian function as follows,
	\[
	\mathcal{L}(a_1,a_2,z_1,z_2,\mu_1,\mu_2)=\mathscr{T}{\bm b}{\bm a}{\bm a}+\mu_1\left(1-z_1^2-z_2^2\right)+\mu_2\left(1-a_1^2-a_2^2\right).
	\]
	By the optimality conditions for constrained optimization problems,  we may	take partial derivatives with respect to $z_1, z_2$ and set them to zero to yield (noticing \eqref{eq:31})
	\[
	\begin{cases}
		\frac{\partial\mathcal{L}}{\partial z_1}=a_2Q_1({\bm a})-a_1Q_2({\bm a})-2\mu_1z_1=0 ,\\
		\frac{\partial\mathcal{L}}{\partial z_2}=Q_3({\bm a})-2\mu_1z_2=0.
	\end{cases}
	\]
		Solve the system of such  equations to obtain 
	\[
	z_1=\frac{a_2Q_1({\bm a})-a_1Q_2({\bm a})}{2\mu_1}, \quad z_2=\frac{Q_3({\bm a})}{2\mu_1}.
	\]
	
	By substituting  them into the unit constraint $z_1^2+z_2^2=1$, we have
	\[
	\mu_1=\pm\frac{1}{2}\sqrt{\left(a_2Q_1({\bm a})-a_1Q_2({\bm a})\right)^2+(Q_3({\bm a}))^2}.
	\]
	
Let's put $\mu_1,\ z_1$ and $z_2$ into  the objective function $\mathscr{T}{\bm b}{\bm a}{\bm a}$ and simplify it. Then we obtain 
	\[
	\mathscr{T}{\bm b}{\bm a}{\bm a}=\pm\sqrt{\left(a_2Q_1({\bm a})-a_1Q_2({\bm a})\right)^2+Q_3({\bm a})^2}.
	\]
	
Taking the maximum value with respect to $\bm a$, as required.
\qed

\begin{theorem}\label{thm:32}
	Let $\mathscr{T}$ be a nonlinear susceptibility tensor.  Assume that $R_{ij}({\bm a})$ is defined in Eq. \eqref{eq:32} for $i,j=1,2$.  Then in uniaxial crystals, the maximum model \eqref{eq:28} of effective SHG coefficient may cut the number of variables in half to $2$. That is, 
	\begin{equation*}\label{eq:34}
		\aligned &d_{\text{eff}}^2=\max\{\mathscr{T}{\bm a}{\bm b}{\bm b}: \ a_1^2+a_2^2=z_1^2+z_2^2=1\mbox{ and }a_3=0 \}\\
		=&\dfrac12\max\left\{R_{11}({\bm a})+R_{22}({\bm a})+\sqrt{\left(R_{11}({\bm a})-R_{22}({\bm a})\right)^2+4R_{12}({\bm a})^2}: a_1^2+a_2^2=1\right\}.\endaligned
	\end{equation*}
\end{theorem}

{\bf Proof.} It follows from the constrained optimization model \eqref{eq:28} and Eq. \eqref{eq:32}  that the objective function $\mathscr{T}{\bm a}{\bm b}{\bm b}$ may be rewritten as follows,
\[
\mathscr{T}{\bm b}{\bm a}{\bm a}=z^2_1R_{11}({\bm a})+2z_1z_2R_{12}({\bm a})+z^2_2R_{22}({\bm a}),
\]
which a quadratic form given by the matrix,$$M({\bm a})=\begin{pmatrix}
	R_{11}({\bm a}) & R_{12}({\bm a})\\
	R_{12}({\bm a})& R_{22}({\bm a})
\end{pmatrix}.$$
By Lemma \ref{lem:21},  for each fixed ${\bm a},$ we have
$$\aligned &\max\{{\bm z}^\top M({\bm a}){\bm x}: z_1^2+z_2^2=1\} \\ =& \dfrac{R_{11}({\bm a})+R_{22}({\bm a})+\sqrt{\left(R_{11}({\bm a})-R_{22}({\bm a})\right)^2+4(R_{12}({\bm a}))^2}}2.\endaligned$$
By this time, taking the maximum value with respect to $\bm a$, as required.
\qed

\begin{theorem}[C-Eigenvalue Inequality] \label{thm:33} Let $\mathscr{T}=\overleftrightarrow{\bm\chi}^{(2)}\in\mathbb{R}^{3\times3\times3}$ be a nonlinear susceptibility tensor in uniaxial crystals, and let $\lambda_{\text{max}}(\mathscr{T})$ be the largest C-eigenvalue of  $\mathscr{T}$. Then the following inequalities hold:
	\[
	d_{\text{eff}}^1\leq\lambda_{\text{max}}(\mathscr{T}),\quad d_{\text{eff}}^2\leq\lambda_{\text{max}}(\mathscr{T}).
	\]
	And the equalities generally do not hold.
\end{theorem}
{\bf Proof.} Since each nonlinear susceptibility tensor is a piezoelectric-type tensor, then from Theorem \ref{thm:21}, it follows that 
$$\aligned\lambda_{\text{max}}(\mathscr{T})=& \max \{\mathscr{T}{\bm x}{\bm y}{\bm y}: {\bm x}^\top {\bm x}={\bm y}^\top  {\bm y}=1\}\\
=& \max \{\mathscr{T}{\bm y}{\bm x}{\bm x}: {\bm x}^\top {\bm x}={\bm y}^\top  {\bm y}=1\}. \endaligned$$
Let $C=\{({\bm x}^\top,{\bm y}^\top)\in\mathbb{R}^3\times\mathbb{R}^3: {\bm x}^\top {\bm x}={\bm y}^\top  {\bm y}=1\}$.  Then the feasible regions of the optimization models \eqref{eq:25} and \eqref{eq:26} are $$D_1=\{({\bm y}^\top,{\bm x}^\top)\in C:  y_3=0,{\bm x}^\top{\bm y}=0\},$$ and  $$D_2=\{({\bm x}^\top,{\bm y}^\top)\in C:  x_3=0,{\bm x}^\top{\bm y}=0\},$$ respectively.  Obviously, $D_1\subset C$ and $D_2\subset C$.
	So the conclusions follow.
\qed

\section{Examples in Uniaxial Crystals}
Let $S=\{{\bm x}\in\mathbb{R}^3: {\bm x}^\top {\bm x}=x_1^2+x_2^2=1, x_3=0\}$ in the sequal.
\begin{example}[$\bar{4}2$m Crystal Class (KH$_2$PO$_4$)]
The non-zero components of the susceptibility tensor of the $\bar{4}2$m crystal class satisfy:
\[
T_{123}=T_{132}=T_{213}=T_{231}=T_{312}=T_{321}=\chi_{14},
\]
and all other components are zero.
\end{example}

{\bf Proof.}  Let ${\bm a}^\top {\bm a}=a_1^2+a_2^2=1$. Obviously, we have 
	\[
	Q_1({\bm a})=Q_2({\bm a})=0,
	\]
	\[
	Q_3({\bm a})=T_{312}a_1a_2+T_{321}a_2a_1=2\chi_{14}a_1a_2.
	\]
	Then
	$$
	d_{\text{eff}}^1=\max_{{\bm a}\in S}\sqrt{0+(2\chi_{14}a_1a_2)^2}=\max_{{\bm a}\in S}2|a_1||a_2||\chi_{14}|=\max_{{\bm a}\in S}(a_1^2+a_2^2)|\chi_{14}|,
$$	
and hence, $d_{\text{eff}}^1=|\chi_{14}|$.
	
	Now we calculate the second-type effective SHG coefficient. Obviously,
\[
R_{11}({\bm a})=0,\quad R_{12}({\bm a})=\chi_{14}(a_2^{2}-a_1^{2}),\quad R_{22}({\bm a})=0.
\]
Then
$$ \aligned
d_{\text{eff}}^2=&\dfrac12\max_{{\bm a}\in S}\left\{0+0+\sqrt{\left(0-0\right)^2+4\chi_{14}^{2}(a_2^{2}-a_1^{2})^2}\right\}\\
=&\dfrac12\max_{{\bm a}\in S}2|\chi_{14}||(a_2^{2}-a_1^{2})|
=\max_{{\bm a}\in S}|a_2^{2}-a_1^{2}||\chi_{14}|,\endaligned
$$
and hence, $d_{\text{eff}}^2=|\chi_{14}|$.

By Theroem \ref{thm:21}, the largest C-eigenvalue of susceptibility tensor may be yielded as follows,
$$\aligned\tilde{\lambda}=&\mathscr{T}\tilde{{\bm x}}\tilde{{\bm y}}\tilde{{\bm y}}\\
=& \max \{\mathscr{T}{\bm x}{\bm y}{\bm y}: {\bm x}^\top {\bm x}=1 ,\ 	{\bm y}^\top  {\bm y}=1\}\\
=&\max_{\|{\bm x}\|=1\atop\|{\bm y}\|=1}  2\chi_{14}(x_1y_2y_3 + x_2y_1y_3 + x_3y_1y_2)\\
=&\max_{\|{\bm y}\|=1}  2\chi_{14}\sqrt{y_1^2y_2^2+y_1^2y_3^2+y_2^2y_3^2}, \endaligned$$
Under the constraint $\|{\bm y}\|^2=y_1^2+y_2^2+y_3^2=1$, the maximum of
$y_1^2y_2^2+y_1^2y_3^2+y_2^2y_3^2$ is $\displaystyle\frac{1}{3}$
(when $y_1=y_2=y_3=\displaystyle\frac{1}{\sqrt{3}}$). Thus,
\[
\lambda_{\max} = \frac{2\sqrt{3}}{3}|\chi_{14}|,
\]
and then
\[ d_{\text{eff}}^1=d_{\text{eff}}^2<\lambda_{\text{max}}(C).
\]

\begin{example}[$4$mm Crystal Class (LiNbO$_3$)]
	The susceptibility  tensor of this crystal class has two independent non-zero components,
	\[
	T_{131}=T_{113}=T_{223}=T_{232}=T_{311}=T_{322}=\chi_{15},\quad T_{333}=\chi_{33}.
	\]
	and all other components are zero.
\end{example}

{\bf Proof.}  Let ${\bm a}^\top {\bm a}=a_1^2+a_2^2=1$. Obviously, we have
\[
Q_1({\bm a})=Q_2({\bm a})=0,
\quad
Q_3({\bm a})=\chi_{15}(a_{1}^{2}+a_{2}^{2}).
\]
Then
$$
d_{\text{eff}}^1=\max_{{\bm a}}|\chi_{15}(a_{1}^{2}+a_{2}^{2})|=|\chi_{15}| ,
$$	
and hence, $d_{\text{eff}}^1=|\chi_{15}|$.

Secondly,we have
\[
R_{11}({\bm a})=0,\quad R_{12}({\bm a})=0,\quad R_{22}({\bm a})=0.
\]
Then,
$d_{\text{eff}}^2=0$.

By Theroem \ref{thm:21}, the largest C-eigenvalue of susceptibility tensor may be established
$$\aligned\tilde{\lambda}=&\mathscr{T}\tilde{{\bm x}}\tilde{{\bm y}}\tilde{{\bm y}}\\
=& \max \{\mathscr{T}{\bm x}{\bm y}{\bm y}: {\bm x}^\top {\bm x}=1 ,\ 	{\bm y}^\top  {\bm y}=1\}\\
=&\max_{\|{\bm x}\|=1\atop \|{\bm y}\|=1}(\chi_{15}(2x_1y_1y_3+2x_2y_2y_3+x_3(y_1^2+y_2^2))+\chi_{33}x_3y_3^2)\\
=&\max_{\|{\bm y}\|=1} \sqrt{(2\chi_{15}y_1y_3)^2+(2\chi_{15}y_2y_3)^2
	+\left[\chi_{15}(y_1^2+y_2^2)+\chi_{33}y_3^2\right]^2}. \endaligned$$

Let $t=y_3^2$, so $y_1^2+y_2^2=1-t$. The maximum occurs at $t=0$ or $t=1$,
corresponding to $\lambda=|\chi_{15}|$ or $|\chi_{33}|$. Thus:
\[
\lambda_{\max} = \max\left\{|\chi_{15}|,|\chi_{33}|\right\}.
\]
So:
\[
d_{\text{eff}}^1=|\chi_{15}|\leq\lambda_{\text{max}}(C),~d_{\text{eff}}^2=0<\lambda_{\text{max}}(C)
.\]
\qed	

\begin{example}[$4$ Crystal Class (Urea)]
	The non-zero components of the susceptibility tensor satisfy,
	\[
	T_{123}=-T_{132}=T_{213}=T_{231}=T_{312}=T_{321}=\chi_{14},
	\]
	\[
	T_{113}=T_{131}=-T_{223}=-T_{232}=T_{311}=-T_{322}=\chi_{15}.
	\]
	and all other components are zero.
\end{example}

{\bf Proof.} Let ${\bm a}^\top {\bm a}=a_1^2+a_2^2=1$. Obviously, we have
\[
Q_1({\bm a})=Q_2({\bm a})=0,
\]
\[
Q_3({\bm a})=\chi_{15}(a_{1}^{2}-a_{2}^{2})+2\chi_{14}a_1a_2.
\]
Then
$$
d_{\text{eff}}^1=\max_{{\bm a}}|\chi_{15}(a_{1}^{2}-a_{2}^{2})+2\chi_{14}a_1a_2|=\sqrt{\chi_{14}^{2}+\chi_{15}^{2}},
$$	
and hence, $d_{\text{eff}}^1=\sqrt{\chi_{14}^{2}+\chi_{15}^{2}}$.
Secondly,we have
\[
R_{11}({\bm a})=0,\quad R_{12}({\bm a})=2\chi_{15}a_{1}a_{2}+\chi_{14}(a_2^{2}-a_1^{2}),\quad R_{22}({\bm a})=0.
\]
Then
$$
d_{\text{eff}}^2=\dfrac12\max_{{\bm a}}\left\{\sqrt{4R_{12}({\bm a})^{2}}\right\}
=\max_{{\bm a}}|R_{12}({\bm a})|=\sqrt{\chi_{14}^{2}+\chi_{15}^{2}},
$$
and hence, $d_{\text{eff}}^2=\sqrt{\chi_{14}^{2}+\chi_{15}^{2}}$.

By Theroem \ref{thm:21}, the largest C-eigenvalue of susceptibility tensor may be obtained as follows, 
$$\aligned\tilde{\lambda}
=& \max \{\mathscr{T}{\bm x}{\bm y}{\bm y}: {\bm x}^\top {\bm x}=1 ,\ 	{\bm y}^\top  {\bm y}=1\}\\
=&\max_{\|\boldsymbol{y}\|=1} \sqrt{(2\chi_{15}y_1y_3)^2
	+\left(2\chi_{14}y_1y_3-2\chi_{15}y_2y_3\right)^2
	+\left[2\chi_{14}y_1y_2-\chi_{15}(y_1^2+y_2^2)\right]^2}
. \endaligned$$

Substitute the symmetric solution $y_1=y_2=y_3=\displaystyle\frac{1}{\sqrt{3}}$ and simplify 
\[
\lambda_{\max} = \sqrt{\chi_{14}^2+\chi_{15}^2+\sqrt{(\chi_{14}^2+\chi_{15}^2)^2+(\chi_{14}\chi_{15})^2}}
.\]

So,
\[
d_{\text{eff}}^1=d_{\text{eff}}^2=\sqrt{\chi_{14}^{2}+\chi_{15}^{2}}=\lambda_{\text{max}}(C).
\]
\qed

\begin{example}[$62$m Crystal Class (Benitoite)]
The non-zero components of the susceptibility tensor of the 62m crystal class satisfy,
	\[
	a_{222}=-a_{112}=-a_{121}=-a_{211}=\chi_{22}.
	\]
	and all other components are zero.
\end{example}

{\bf Proof.} Let ${\bm a}^\top {\bm a}=a_1^2+a_2^2=1$. Obviously, we have
\[
Q_1({\bm a})=-2\chi_{22}a_{1}a_{2},
\quad Q_2({\bm a})=\chi_{22}(a_{2}^{2}-a_{1}^{2}),
\quad
Q_3({\bm a})=0.
\]
Then
$$ \aligned d_{\text{eff}}^1=&\max_{{\bm a}}\sqrt{(-2a_{2}\chi_{22}a_{1}a_{2}-a_{1}\chi_{22}a_{2}^{2}+a_{1}\chi_{22}a_{1}^{2})^{2}} \\
=& |\chi_{22}|\max_{{\bm a}}|-3a_{1}a_{2}^{2}+a_{1}^{3}|
=|\chi_{22}|.\endaligned
$$	
and hence, $d_{\text{eff}}^1=|\chi_{22}|$.
Secondly,we have
\[
R_{11}({\bm a})=\chi_{22}(3a_{1}^{2}a_{2}-a_{2}^{3}),\quad R_{12}({\bm a})=0,\quad R_{22}({\bm a})=0.
\]
Then
$$
d_{\text{eff}}^2=\dfrac12\max_{{\bm a}\in S}\left\{R_{11}({\bm a})+|R_{11}({\bm a})|\right\}
=\max_{{\bm a}\in S}|\chi_{22}(3a_{1}^{2}a_{2}-a_{2}^{3})|=|\chi_{22}|,
$$
and hence, $d_{\text{eff}}^2=|\chi_{22}|$.

By Theroem \ref{thm:21}, the largest C-eigenvalue of susceptibility tensor may be yielded as follows,
$$\aligned\tilde{\lambda}=&\mathscr{T}\tilde{{\bm x}}\tilde{{\bm y}}\tilde{{\bm y}}\\
=& \max \{\mathscr{T}{\bm x}{\bm y}{\bm y}: {\bm x}^\top {\bm x}=1 ,\ 	{\bm y}^\top  {\bm y}=1\}\\
=&\max_{\|\boldsymbol{y}\|=1} \sqrt{(-2\chi_{22}y_1y_2)^2
	+\left[\chi_{22}(y_2^2-y_1^2)\right]^2}\\
=& |\chi_{22}|. \endaligned$$

Under $y_1^2+y_2^2+y_3^2=1$, $y_1^2+y_2^2 \le 1$ and reaches 1 when $y_3=0$. Thus:
\[
\lambda_{\max} = |\chi_{22}|.
\]

So,
\[
d_{\text{eff}}^1=d_{\text{eff}}^2=\lambda_{\text{max}}(C)=|\chi_{22}|.
\]
\qed

\begin{example}[$6$ Crystal Class ($\alpha$-LiIO$_{3}$)]
	The non-zero components of the susceptibility tensor satisfy:
	\[
	T_{111}=-T_{122}=-T_{212}=-T_{221}=\chi_{11}, \]
	\[
	T_{222}=-T_{112}=-T_{121}=-T_{211}=\chi_{22}.
	\]
	and all other components are zero.
\end{example}

{\bf Proof.} Let ${\bm a}^\top {\bm a}=a_1^2+a_2^2=1$. Obviously, we have

\[
Q_1({\bm a})=\chi_{11}(a_{1}^{2}-a_{2}^{2})-2\chi_{22}a_{1}a_{2},
\]
\[
Q_2({\bm a})=\chi_{22}(a_{2}^{2}-a_{1}^{2})-2\chi_{11}a_{1}a_{2},
\]
\[
Q_3({\bm a})=0.
\]
Then
$$ \aligned d_{\text{eff}}^1=& \max_{{\bm a}\in S}\sqrt{(\chi_{11}(3a_{1}^{2}a_{2}-a_{2}^{3})-\chi_{22}(3a_{1}a_{2}^{2}-a_{1}^{3}))^{2}}\\
=& \max_{{\bm a}\in S}|(\chi_{11}(3a_{1}^{2}a_{2}-a_{2}^{3})-\chi_{22}(3a_{1}a_{2}^{2}-a_{1}^{3})|\\
=& \sqrt{\chi_{11}^{2}+\chi_{22}^{2}}.\endaligned
$$	
and hence, $d_{\text{eff}}^1=\sqrt{\chi_{11}^{2}+\chi_{22}^{2}}$.
Secondly,we have
\[
R_{11}({\bm a})=\chi_{11}(a_{1}a_{2}^{2}+a_{1^{3}})+\chi_{22}(3a_{1}^{2}a_{2}-a_{2}^{3}),\ R_{12}({\bm a})=0,\ R_{22}({\bm a})=0.
\]
Then
$$
d_{\text{eff}}^2=\dfrac12\max_{{\bm a}\in S}\left\{R_{11}({\bm a})+|R_{11}({\bm a})|\right\}
=\max_{{\bm a}\in S}|R_{11}({\bm a})|=\max(|\chi_{11}|,|\chi_{22}|),
$$
and hence, $d_{\text{eff}}^2=\max(|\chi_{11}|,|\chi_{22}|)$.

By Theroem \ref{thm:21}, the largest C-eigenvalue of susceptibility tensor may be yielded as follows,
$$\aligned\tilde{\lambda}
=& \max \{\mathscr{T}{\bm x}{\bm y}{\bm y}: {\bm x}^\top {\bm x}=1 ,\ 	{\bm y}^\top  {\bm y}=1\}\\
=& \max_{\|\boldsymbol{y}\|=1} \sqrt{
	\left[\chi_{11}(y_1^2-y_2^2)+(\chi_{22}-\chi_{11})y_1y_2\right]^2
	+\left[\chi_{22}(y_2^2-y_1^2)+(\chi_{11}-\chi_{22})y_1y_2\right]^2}. \endaligned$$

Let $y_1=\cos\theta$, $y_2=\sin\theta$.
Thus,
\[
\lambda_{\max} =\sqrt{\chi_{11}^{2}+\chi_{22}^{2}},
\]
and so, we have 
\[
d_{\text{eff}}^1=\lambda_{\text{max}}(C),~d_{\text{eff}}^2\leq\lambda_{\text{max}}(C).
\]
\qed

\section{Conclusion}
In this paper, the computability of effective SHG coefficient is discussed  for nonlinear optical susceptibility tensors in uniaxial crystals.  Firstly, the optimization models  are established for the calculation of  effective SHG coefficient.  Secondly, we study two optimization models and cut the number of variables of  such  maximum models  in half to $2$. Moreover, we also present a comparison between the effective SHG coefficient and C-eigenvalue of susceptibility tensor.  Finally, through calculations of typical crystal classes, the correctness and universality of the theory are verified, providing a new theoretical tool for the study of phase matching in nonlinear optics.
\section*{Declaration of competing interest}
The authors declare that they have no known competing financial interests or personal rela-tionships that could have appeared to influence the work reported in this paper.
\section*{Authors' contributions}These authors contributed equally to this work.
\section*{Availability of data and materials}
This manuscript has no associated data or the data will not be deposited. [Authors' comment: This is a theoretical study and there are no external data associated with the manuscript].

\end{document}